\documentclass{amsart}  % Replace amsproc by the name of the author package.

\usepackage{graphicx}
\usepackage{bm}
\usepackage{array}
\usepackage{multicol}
\usepackage[usenames,dvipsnames,table]{xcolor}
\usepackage{hyperref}
\usepackage{tikz}
\usetikzlibrary{calc,arrows,shapes,positioning}

\hypersetup{
pdfauthor={Chris Magyar and Ursula Whitcher},
pdftitle={Strong arithmetic mirror symmetry and toric isogenies},
pdfcreator={LaTeX with hyperref package},
pdfproducer={pdftex},
colorlinks={false}}

\newtheorem{theorem}{Theorem}[section]
\newtheorem{conjecture}[theorem]{Conjecture}
\newtheorem{question}[theorem]{Question}
\newtheorem{lemma}[theorem]{Lemma}
\theoremstyle{definition}
\newtheorem{definition}[theorem]{Definition}
\newtheorem{example}[theorem]{Example}

\theoremstyle{remark}

\numberwithin{equation}{section}

\newcolumntype{C}[1]{>{\centering\let\newline\\\arraybackslash}m{#1}}

\newcommand{\p}[1]{\left({#1}\right)}
\newcommand{\Z}{\mathbb{Z}}
\newcommand{\Q}{\mathbb{Q}}

\newcommand{\C}{\mathbb{C}}
\newcommand{\F}{\mathbb{F}}

\newcommand{\derive}[2]{\frac{\partial^{#2} #1}{\partial t^{#2}}}
\newcommand{\of}[1]{\!\left({#1}\right)}

\tikzset{view1/.style={x={(0cm,0.45cm)},y={(-0.35cm, -0.4cm)},z={(0.5cm, -0.2cm)}}}
\tikzset{axis/.style={color=black!25,font=\scriptsize,thin,->}}
\tikzset{axis2/.style={color=black,font=\scriptsize,thin,->}}
\tikzset{label/.style={font=\footnotesize}}
\tikzset{back/.style={dashed, color=blue, thick}}
\tikzset{edge/.style={color=blue, thick}}
\tikzset{vertex/.style={coordinate,inner sep=1.6pt,circle,draw=Black,fill=Red,thick,anchor=base}}
\tikzset{point/.style={coordinate,inner sep=1.1pt,circle,draw=Black,fill=Red,thick,anchor=base}}
\tikzset{point2/.style={coordinate,inner sep=1.1pt,circle,draw=Black,fill=Black,thick,anchor=base}}

\tikzset{shade1/.style={fill=green!50,fill opacity=0.5}}
\tikzset{shade2/.style={fill=green!66,fill opacity=0.5}}
\tikzset{shade3/.style={fill=green!90,fill opacity=0.5}}
\tikzset{shade4/.style={fill=green,fill opacity=0.5}}

\newcommand{\draworigin}{\node[vertex] at (0,0,0) {};}

\newcommand{\draworiginb}{\draw[thick] (-0.15,0) -- (0.15,0);
\draw[thick] (0,0.15) -- (0,-0.15);}

\newcommand{\pdual}{
\begin{tikzpicture}[]
\draw[thick,<->] (-0.4,0) -- (0.4,0);
\node[above] at (0,-0.05) {\footnotesize$\Delta\!^\circ$};
\end{tikzpicture}}

\begin{document}

\title{Strong arithmetic mirror symmetry and toric isogenies}

% author one information
\author{Christopher Magyar}
\address{University of Wisconsin--Eau Claire}
%\curraddr{}
\email{magyarca@uwec.edu}
%\thanks{}

%  author two information
\author{Ursula Whitcher}
\address{University of Wisconsin--Eau Claire and Mathematical Reviews}
%\curraddr{}
\email{whitchua@uwec.edu}
\thanks{We thank Volker Braun and the participants of Sage Days 50 for many interesting conversations, the AIM SQuaRE program for an inspirational research environment, and the anonymous referees for suggestions which materially improved our exposition.  The first author was supported in part by Student Blugold Commitment Differential Tuition funds through the University of Wisconsin-Eau Claire Student/Faculty Research Collaboration program.}

%\subjclass[2000]{Primary }
%   The 2010 edition of the Mathematics Subject Classification is
%   now available.  If you are citing a classification from the
%   new scheme, use the following input coding instead.

\date{}

\begin{abstract}
We say a mirror pair of Calabi-Yau varieties exhibits strong arithmetic mirror symmetry if the number of points on each variety over a finite field is equivalent, modulo the order of that field.  We search for strong mirror symmetry in pencils of toric hypersurfaces generated using polar dual pairs of reflexive polytopes.  We characterize the pencils of elliptic curves where strong arithmetic mirror symmetry arises, and provide experimental evidence that the phenomenon generalizes to higher dimensions.  We also provide experimental evidence that pencils of K3 surfaces with the same Picard-Fuchs equation have related point counts.
\end{abstract}

\maketitle

\section{Introduction}

Mathematical formulations of mirror symmetry have led to deep insights regarding relationships between very different geometric spaces.  In some cases, mirror pairs of varieties have related arithmetic structure.  

Let $X/\F_q$ be an algebraic variety over the finite field of order $q$, and let $\#X(\F_{q^s})$ be the number of $\F_{q^s}$-rational points on $X$.  Then the \emph{zeta function} of $X$ is the generating function given by

\[Z(X/\F_q,T):=\exp\left(\sum_{s=1}^{\infty}\#X(\F_{q^s}) \frac{T^s}{s}\right)\in \Q[[T]].\]

The celebrated Weil conjectures, proved in \cite{dwork}, state that the zeta function is a rational function of $T$, with the degrees of the numerator and denominator determined by the Betti numbers of $X$.  Mirror symmetry for Calabi-Yau threefolds entails an interchange of the Hodge numbers $h^{1,1}$ and $h^{2,1}$, and thus a relationship between the Betti numbers of mirror Calabi-Yau threefolds.  Thus, we may use information about the structure of the zeta function of a Calabi-Yau threefold to make predictions about the structure of the zeta function of its mirror.  

The authors of \cite{CORV, CORV2} identified a much stronger relationship between zeta functions for a specific mirror family of Calabi-Yau threefolds.  Any smooth homogeneous quintic polynomial in five variables defines a Calabi-Yau threefold as a hypersurface in $\mathbb{P}^4$.  We may construct the mirror family to these quintics using the Greene-Plesser construction, as follows.  Consider the Fermat quintic pencil $X_t$ given by
\[x_0^5 + x_1^5 + x_2^5 + x_3^5 + x_4^5 - 5 t x_0 x_1 x_2 x_3 x_4 = 0.\]
Each member of the pencil admits a group action by $(\mathbb{Z}/5\mathbb{Z})^3$, which consists of multiplying each coordinate by suitably chosen fifth roots of unity.  Taking the quotient by this group action and choosing a minimal resolution of singularities, we obtain a new pencil of Calabi-Yau threefolds, $Y_t$.  Then $Y_t$ is the mirror family to the family of smooth quintics in $\mathbb{P}^4$.  The computations in \cite{CORV2} show that $Z(X_t/\F_q,T)$ and $Z(Y_t/\F_q,T)$ share a common quartic factor (depending on $t$) in their numerator.

One may try to extend the arithmetic mirror symmetry relationships observed for the Fermat quintic pencil in two ways: by considering the Fermat pencil and its Greene-Plesser mirror in other dimensions, or by considering other mirror constructions.  

Let $X_t$ be the Fermat quartic pencil.  In this case, the zeta function and mirror zeta function have been described explicitly.  In \cite{kadir}, building on results of Dwork, the zeta function is factored as:

\[Z(X_t/\F_p,T)=\frac{1}{(1-T)(1-pT)(1-p^2T)R_t(T) Q_t(T)}\]

\noindent Here $R_t(T)$ is a degree 3 polynomial of the form $(1\pm pT)(1-a_t T+p^2T)$, where $a_t$ is a constant, and $Q_t$ is a polynomial of degree 18.

Let $Y_t$ be the mirror family to quartics in $\mathbb{P}^3$ (constructed using Greene-Plesser and the Fermat pencil).  Then \cite{kadir} gives the mirror zeta function as:
\[Z(Y_t/\F_p,T)=\frac{1}{(1-T)(1-pT)^{19}(1-p^2T)R_t(T)}.\] 

\noindent The factor $R_t(T)$ corresponds to periods of the holomorphic form and its derivatives, and is invariant under mirror symmetry.  Furthermore, by applying Tate's conjecture one can show that for fields $\F_q$ containing sufficiently many roots of unity, $Z(X_t/\F_q,T) = Z(Y_t/\F_q,T)$ (see \cite{square}).

In \cite{wan}, Wan characterized the arithmetic mirror symmetry phenomenon that the authors of \cite{CORV2} observed for the Fermat quintic pencil as \emph{strong arithmetic mirror symmetry}.  In Wan's formulation, strong mirror symmetry arises when one has one-parameter families of Calabi-Yau varieties $X_t$ and $Y_t$ such that $X_t$ and $Y_t$ are mirrors for each $t$ where both are smooth.  In particular, if we generalize the Greene-Plesser construction to smooth $n+1$-folds in $\mathbb{P}^n$, we obtain strong mirrors in this sense.  Wan showed that these families of $n+1$-folds satisfy an arithmetic relationship; we take this relationship as our definition of strong arithmetic mirror symmetry.

\begin{definition}
Let $X$ and $W$ be a mirror pair of smooth Calabi-Yau varieties defined over $\mathbb{Q}$.  Suppose that for any prime $p$ where $X$ and $W$ have good reduction and any finite field $\F_q$ of characteristic $p$ and order $q$, $X$ and $W$ satisfy
\[\#X(\F_q) = \#W(\F_q) \pmod{q}.\]
Then we say $X$ and $W$ exhibit \emph{strong arithmetic mirror symmetry} and are a \emph{strong mirror pair}.
\end{definition}

Work of \cite{adolphsonsperber} and \cite{yu} on Dwork's unit root shows that the point counts $\pmod{q}$ for the Fermat pencil are related to solutions of the Picard-Fuchs equations for the family.  The Picard-Fuchs equation for the mirror pencil is identical to the Fermat pencil's Picard-Fuchs equation, but whether \cite{yu} and \cite{adolphsonsperber}'s results on unit roots can be generalized to give a new proof of strong arithmetic mirror symmetry for this family is not known.  One of our goals in the present work is to provide further evidence for a relationship between Picard-Fuchs equations and point counts.
% Igusa's results on number of rational points (mod order of field) for elliptic curves, Picard-Fuchs equations.  

What is known about arithmetic mirror symmetry for other mirror constructions?  There is a natural extrapolation of the Greene-Plesser mirror construction to deformations of diagonal hypersurfaces in weighted projective spaces.  Kadir, inspired by calculations for a specific two-parameter family of Calabi-Yau threefolds in a weighted projective space, argued that the zeta functions of any Greene-Plesser mirror pair of Calabi-Yau threefolds will have a shared common factor in their numerator in \cite{kadirpaper}.  Motivated by the Greene-Plesser construction, Fu and Wan showed in \cite{fw} that for appropriately chosen groups $G$,
\[\#X(\F_q) = \#(X/G)(\F_q) \pmod{q}.\]

\noindent However, these results for orbifolds do not extend to the mirror family, where singularities have been resolved.  More recent investigations of arithmetic mirror symmetry have focused on the Berglund-H\"{u}bsch-Krawitz mirror construction (cf. \cite{square, ap}).  This allows an extension of observations made about the mirrors of diagonal hypersurfaces to the mirrors of certain other hypersurfaces in weighted projective spaces.

In the current work, we generalize Fermat pencils in projective space to certain one-parameter families in toric varieties, using Batyrev's mirror symmetry construction.  From the toric point of view, the Greene-Plesser mirror of the Fermat pencil is determined by choosing monomials corresponding to the vertices and origin of the polytope for $\mathbb{P}^n$, and the Fermat pencil itself is determined by choosing the vertex and origin monomials of the polar dual polytope.  We use the vertices and the origin in polar dual pairs of reflexive polytopes to construct pairs of pencils of hypersurfaces, and experimentally investigate when these pairs exhibit strong arithmetic mirror symmetry, testing small prime numbers.  We find evidence for strong mirror symmetry when the pairs of pencils are related by resolutions of finite quotient maps, in a further generalization of the Greene-Plesser mirror construction.

For elliptic curves, we apply the classical Tate's isogeny theorem to describe the appropriate strong arithmetic mirror pairs completely.  We identify new mirror pairs of K3 surface pencils which exhibit strong arithmetic mirror symmetry for small primes, and describe conjectural, combinatorial criteria for finding similar examples in higher dimensions.  We show that our K3 surface pencils are of high Picard rank.  The non-weighted projective space examples fall naturally into two groups with the same Picard-Fuchs equations, which we use to study the geometric structure of these families.  Experimentally, we detect relationships between point counts for different K3 surface pencils in the same group, even when these pencils were not generated as a mirror pair.  Because members of the same group are related by resolutions of quotients by finite abelian groups, we may view the point-counting relationship as a higher-dimensional generalization of Tate's isogeny theorem.

\section{Hypersurfaces in toric varieties}\label{S:toric}

We begin by reviewing the construction of Calabi-Yau varieties as toric hypersurfaces, in order to fix notation.  Let $N$ be a lattice isomorphic to $\mathbb{Z}^n$.  A \emph{cone} in $N$ is a subset of the real vector space $N_\mathbb{R} = N \otimes \mathbb{R}$ generated by nonnegative $\mathbb{R}$-linear combinations of a set of vectors $\{v_1, \dots , v_m\} \subset N$.  We assume that cones are strongly convex, that is, they contain no line through the origin.  We say a cone is \emph{simplicial} if its generators are linearly independent over $\mathbb{R}$.  A \emph{fan} $\Sigma$ consists of a finite collection of cones such that each face of a cone in the fan is also in the fan, and any pair of cones in the fan intersects in a common face.  We say a fan $\Sigma$ in $N_\mathbb{R}$ is \emph{complete} if the union of all of the cones in $\Sigma$ is the entire vector space $N_\mathbb{R}$, and we say a fan is simplicial if all of its cones are simplicial.

Given a complete fan $\Sigma$, we may construct a toric variety in the following way.  Let $\Sigma(1)=\{v_1, \dots, v_q\}$ be lattice points generating the one-dimensional cones of $\Sigma$.  To each lattice point $v_j$, we associate a coordinate $z_j$ on $\mathbb{C}^q$.
Let $\mathcal{S}$ denote any subset of $\Sigma(1)$ that does \emph{not} span a cone of $\Sigma$, and let $\mathcal{V}(\mathcal{S})\subseteq \C^q$ be the linear subspace defined by setting $z_j = 0$ if the corresponding cone is in $\mathcal{S}$.  Define $Z(\Sigma) = \cup_\mathcal{S} \mathcal{V}(\mathcal{S})$.  Then $(\C^*)^q$ acts on $\C^q - Z(\Sigma)$ by coordinatewise multiplication.  
Let us write $v_j$ in coordinates as $(v_{j 1}, \dots, v_{j n})$.  Let $\phi : (\C^*)^q \to (\C^*)^n$ be given by
\[ \phi(t_1, \dots, t_q) \mapsto \left( \prod_{j=1}^q t_j ^{v_{j1}} , \dots, \prod_{j=1}^q t_j^{v_{j n}} \right) \]
The kernel of $\phi$ is isomorphic to $(\C^*)^{q-n} \times A$, where $A$ is a finite abelian group.  The toric variety $V_\Sigma$ associated with the fan $\Sigma$ is given by
\[ V_\Sigma  = (\C^q - Z(\Sigma)) / \text{Ker}(\phi).\]
\noindent The variables $z_j$ may be viewed as homogeneous coordinates on $V_\Sigma$, subject to the identifications required by $\text{Ker}(\phi)$.

%\begin{example}\label{Ex:P1xP1}
%Let $N$ be a two-dimensional lattice generated by $e_1$ and $e_2$.  Let $\Sigma$ be the fan illustrated in Figure~\ref{F:fan}, which has two-dimensional cones generated by $\{e_1, e_2\}$, $\{e_1, -e_2\}$,  $\{-e_1, e_2\}$, and  $\{-e_1, -e_2\}$.  Then $Z(\Sigma)$ consists of points of the form $(0,0,z_3,z_4)$ or $(z_1,z_2,0,0)$.  We have $V_\Sigma =(\C^4 - Z(\Sigma))/\sim$.  Here, the equivalence relation $\sim$ is given by
%\begin{align*}
%(z_1,z_2,z_3,z_4) &\sim (\lambda_1 z_1,\lambda_1 z_2,z_3,z_4)\\
%(z_1,z_2,z_3,z_4) &\sim (z_1, z_2,\lambda_2 z_3,\lambda_2 z_4)
%\end{align*}
%where $\lambda_1,\lambda_2 \in \C^*$.  Thus, $V_\Sigma = \mathbb{P}^1\times \mathbb{P}^1$.
%\end{example}
%
%\begin{figure}[htb]
%\begin{tikzpicture}[]
%\fill[fill=green!20] (0,0) -- (0,2) -- (2,2) -- (2,0) -- cycle;
%\fill[fill=blue!20] (0,0) -- (0,2) -- (-2,2) -- (-2,0) -- cycle;
%\fill[fill=red!15] (0,0) -- (0,-2) -- (-2,-2) -- (-2,0) -- cycle;
%\fill[fill=yellow!30] (0,0) -- (0,-2) -- (2,-2) -- (2,0) -- cycle;
%\draw[ultra thick, <->] (-2,0) -- (2,0);
%\draw[ultra thick, <->] (0,-2) -- (0,2);
%\node[above] at (0.6,-.05) {$e_1$};
%\node[left] at (0.05,0.6) {$e_2$};
%\end{tikzpicture}
%\caption{The fan corresponding to $\mathbb{P}^1\times \mathbb{P}^1$.}
%\label{F:fan}
%\end{figure}

The dual lattice $M$ of $N$ is given by $\mathrm{Hom}(N, \mathbb{Z})$; it is also isomorphic to $\mathbb{Z}^n$.  We write the pairing of $v \in N$ and $w \in M$ as $\langle v , w \rangle$.  We will use pairs of polytopes in $N_\mathbb{R}$ and $M_\mathbb{R}$ to construct pairs of toric varieties.  A \emph{lattice polytope} $\Delta$ has vertices in a lattice.  Let $\Delta$ be a lattice polytope in $N$ which has the origin as an interior point.  The \emph{polar dual} of $\Delta$ is defined as 
\[\Delta^\circ = \lbrace w \in M_\mathbb{R} \, | \, \langle v,w \rangle \geq -1 \, \forall \, v \in \Delta \rbrace.\] 
If both $\Delta$ and $\Delta^\circ$ are lattice polytopes, we call them \emph{reflexive polytopes}.  We illustrate a two-dimensional pair of reflexive polytopes in Figure~\ref{F:reflexivequads}.

\begin{figure}[htb]
\begin{tabular}{C{2.5cm} C{1cm} C{2.5cm}}
%\begin{tabular}{c C{1cm} c}
\textbf{Polytope 3}& &
\textbf{Polytope 14}\\
\begin{tikzpicture}[]
\draworiginb
\fill[shade1] (-1,0) -- (0,1) -- (1,0) -- (0,-1) -- cycle;
\draw[edge] (-1,0) -- (0,1) -- (1,0) -- (0,-1) -- cycle;
\node[point2] at (-1,-1) {};
\node[vertex] at (-1,0) {};
\node[point2] at (-1,1) {};
\node[vertex] at (0,1) {};
\node[vertex] at (1,0) {};
\node[point2] at (1,1) {};
\node[point2] at (1,-1) {};
\node[vertex] at (0,-1) {};
\end{tikzpicture}& \pdual &
\begin{tikzpicture}[]
\draworiginb
\fill[shade1] (-1,-1) -- (-1,1) -- (1,1) -- (1,-1) -- cycle;
\draw[edge] (-1,-1) -- (-1,1) -- (1,1) -- (1,-1) -- cycle;
\node[vertex] at (-1,-1) {};
\node[point] at (-1,0) {};
\node[vertex] at (-1,1) {};
\node[point] at (0,1) {};
\node[point] at (1,0) {};
\node[vertex] at (1,1) {};
\node[vertex] at (1,-1) {};
\node[point] at (0,-1) {};
\end{tikzpicture}
\end{tabular}
\caption{Dual reflexive quadrilaterals.}
\label{F:reflexivequads}
\end{figure}
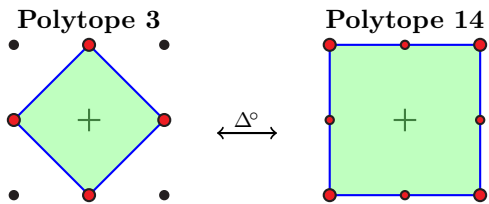

Reflexive polytopes have been classified in two, three, and four dimensions.  Under multiplication by elements of $\mathrm{GL}\!\left(n, \Z \right)$, there are 16, 4319, and 473 800 776 equivalence classes of reflexive polytopes, respectively (see \cite{ks}).

Given a reflexive polytope $\Delta$, we may obtain a fan $R$ by taking cones over the polytope's faces.  The one-dimensional cones of $R$ are then generated by the vertices of $\Delta$, and the $n$-dimensional cones are given by the facets of $\Delta$.  For example, the quadrilateral labeled Polytope~3 in Figure~\ref{F:reflexivequads} corresponds to the fan that has two-dimensional cones generated by $\{e_1, e_2\}$, $\{e_1, -e_2\}$,  $\{-e_1, e_2\}$, and  $\{-e_1, -e_2\}$.  The resulting toric variety is $\mathbb{P}^1 \times \mathbb{P}^1$.  

Computationally, we may identify $\text{Ker}(\phi)$ using the matrix $\mathrm{Mat}_\Delta$ whose columns are given by the vertices of $\Delta$.  The action of $(\C^*)^{q-n}$ is given by the kernel of $\mathrm{Mat}_\Delta$ as an integer matrix, multiplying by $\mathrm{Mat}_\Delta$ on the left.  The free abelian group $A$ is given by the elementary divisors of $\mathrm{Mat}_\Delta$.  (These are the entries of the Smith normal form on the main diagonal.)  For example, if we let $\Delta$ be the quadrilateral labeled Polytope~3 in Figure~\ref{F:reflexivequads} and $\Delta^\circ$ its polar dual, both $\mathrm{Mat}_\Delta$ and $\mathrm{Mat}_{\Delta^\circ}$ have kernels generated by $(1,0,0,1)$ and $(0,1,1,0)$ in $\mathbb{Z}^4$.  However, the elementary divisors of $\mathrm{Mat}_\Delta$ are both $1$, while the elementary divisors of $\mathrm{Mat}_{\Delta^\circ}$ are $1$ and $2$.  As we have already observed, $\Delta$ determines the toric variety $\mathbb{P}^1 \times \mathbb{P}^1$, while $\Delta^\circ$ yields a quotient of $\mathbb{P}^1 \times \mathbb{P}^1$ by a finite abelian group of order $2$.

The fan $R$ given by cones over a polytope's faces will be complete, but will not necessarily be simplicial.  Thus, the corresponding toric variety may have singularities.  However, we can construct a maximal simplicial refinement $\Sigma$ by using all of the lattice points of $\Delta$ to generate one-dimensional cones; such a refinement always exists by results of \cite{op}.  This operation on fans corresponds to resolution of singularities on the level of toric varieties.  Anticanonical hypersurfaces in toric varieties obtained from reflexive polytopes are Calabi-Yau varieties.  If the dimension of the reflexive polytope is at most four, taking a maximal simplicial refinement $\Sigma$ guarantees that the general anticanonical hypersurface is smooth.  Thus, two-dimensional reflexive polytopes determine families of elliptic curves, three-dimensional reflexive polytopes determine families of K3 surfaces, and four-dimensional reflexive polytopes yield families of Calabi-Yau threefolds.  Polar dual pairs of reflexive polytopes correspond to mirror pairs of Calabi-Yau varieties.

\begin{example}\label{Ex:Pn}
Let $e_1, \dots, e_n$ be generators of $N$.  Let $\Delta$ be the simplex with vertices $\{e_1, \dots, e_n, -e_1 + -e_2 + \dots + -e_n\}$.  The toric variety corresponding to the fan over the faces of $\Delta$ is $\mathbb{P}^n$; the anticanonical hypersurfaces are homogeneous polynomials in $n+1$ variables of degree $n+1$.  A maximal simplicial refinement of the fan over the faces of $\Delta^\circ$ yields Calabi-Yau varieties isomorphic to the mirror varieties $Y_t$ obtained from the Greene-Plesser construction.
\end{example}

Any reflexive simplex defines either a weighted projective space or the resolution of a quotient of a weighted projective space by a finite abelian group.  This observation allows us to study weighted projective space generalizations of the Greene-Plesser construction in a toric context.

We may use the lattice points of a reflexive polytope and its polar dual to write an explicit expression for an anticanonical hypersurface in homogeneous coordinates.  Let $\{v_k\} \subset \Delta \cap N$ generate the one-dimensional cones of $\Sigma$, and choose an arbitrary complex number $c_x$ for each lattice point $x$ of $\Delta^\circ$.  Then an anticanonical hypersurface is given by
\begin{equation} f = \sum_{x \in \Delta^\circ \cap M} c_x \prod_{k=1}^q z_k^{\langle v_k, x \rangle + 1}.\end{equation}\label{E:polynomialpencil}

If $\Delta$ is the simplex described in Example~\ref{Ex:Pn}, we may use Equation~\ref{E:polynomialpencil} to specify the Fermat pencil in $\mathbb{P}^n$ by choosing special values of the parameters $c_x$.  Specifically, we set the parameters corresponding to the vertices of $\Delta^\circ$ equal to one, associate a single parameter $t$ to the origin, and set the parameters corresponding to other lattice points to $0$.
Combinatorially, this operation makes sense for any reflexive polytope.  Thus, we define a one-parameter family of polynomials $f_{\Delta, t}$ by:
\begin{equation} f_{\Delta, t} = \left(\sum_{x \; \in \; \mathrm{vertices}(\Delta^\circ)} \prod_{k=1}^q z_k^{\langle v_k, x \rangle + 1}\right) + t \prod_{k=1}^q z_k.\end{equation}
For a polar dual pair of reflexive polytopes $\Delta$ and $\Delta^\circ$, $f_{\Delta, t}$ and $f_{\Delta^\circ, t}$ determine one-parameter subfamilies $X_t$ and $Y_t$ of the mirror families of Calabi-Yau varieties associated to the polytopes.  We investigate when $X_t$ and $Y_t$ possess strong arithmetic mirror symmetry in the following sections.

\section{Elliptic curves}\label{S:elliptic}

Polar dual pairs of two-dimensional reflexive polytopes determine mirror families of elliptic curves.  Recall that an \emph{isogeny} between two elliptic curves $E_1$ and $E_2$ is a (non-constant) rational morphism $\phi: E_1 \to E_2$ that maps the identity element of $E_1$ to the identity element of $E_2$.  Isogenies of elliptic curves are finite surjective maps.  Tate's Isogeny Theorem (\cite{tate}) relates point counts on isogenous curves: two elliptic curves over a finite field $\mathbb{F}_{q}$ are isogenous if and only if they have the same number of $\mathbb{F}_{q}$-rational points.

\begin{theorem}\label{T:polygons}
Let $\Delta$ and $\Delta^\circ$ be a polar dual pair of reflexive polygons, and let $\Sigma$ and $\Sigma^\circ$ be maximal simplicial refinements of the fans over the faces of $\Delta$ and $\Delta^\circ$, respectively.  Let $X_t$ and $Y_t$ be members of the families determined by the polynomials $f_{\Delta, t}$ and $f_{\Delta^\circ, t}$.  Then $\#X_t(\mathbb{F}_{q}) = \#Y_t(\mathbb{F}_{q})$ for each rational $t$ where $X_t$ and $Y_t$ are smooth when one of the following holds:

\begin{enumerate}
\item The polygons $\Delta$ and $\Delta^\circ$ are triangles.  In this case, either $V_\Sigma$ or $V_{\Sigma^\circ}$ is a weighted projective space.
\item The polygon $\Delta$ is self-dual.
\item The polygons $\Delta$ and $\Delta^\circ$ are quadrilaterals, and either $V_\Sigma$ or $V_{\Sigma^\circ}$ is $\mathbb{P}^1 \times \mathbb{P}^1$.
\end{enumerate}

\end{theorem}

\begin{proof}
We examine the sixteen equivalence classes of reflexive polygons directly.  When $\Delta$ and $\Delta^\circ$ are triangles, self-dual, or equivalent to the pair illustrated in Figure~\ref{F:reflexivequads}, then either $V_{\Sigma^\circ}$ is a resolution of a finite quotient of $V_\Sigma$, or $V_\Sigma$ is a finite quotient of $V_{\Sigma^\circ}$.  In this case, the quotient map on toric varieties induces an isogeny between each pair of smooth elliptic curves $X_t$ and $Y_t$.  By Tate's Isogeny Theorem, we have $\#X_t(\mathbb{F}_{q}) = \#Y_t(\mathbb{F}_{q})$.
\end{proof}

The cases described in Theorem~\ref{T:polygons} include 10 of the 16 equivalence classes of reflexive polygons.  In the remaining 6 cases, we used SageMath (\cite{sage}) to count points on $X_t$ and $Y_t$ in the case that $q$ is a prime number less than 100, and $t=0, 1, \dots, q-1$.  We found that $\#X_t(\mathbb{F}_{q})$ and $\#Y_t(\mathbb{F}_{q})$ were different in these cases.  Applying Tate's Isogeny Theorem, we see that there is no isogeny between $X_t$ and $Y_t$.  Thus, Theorem~\ref{T:polygons} characterizes the situations where strong arithmetic symmetry arises for $f_{\Delta, t}$ and $f_{\Delta^\circ, t}$ completely.

\section{Experimental evidence for strong mirror symmetry}

Let us now investigate families of K3 surfaces determined by three-dimensional reflexive polytopes.  We wish to generalize the results of Theorem~\ref{T:polygons}.  Thus, we seek to identify mirror pairs where the corresponding toric varieties are related by a resolution of a finite quotient map.

When we take the fan $R$ over the faces of a polytope $\Delta$ with $q$ vertices, we obtain the toric variety $(\mathbb{C}^q - Z(R))/((\mathbb{C}^*)^{q-n} \times A)$, where $A$ is a finite abelian group.  The zero set $Z(R)$ is determined by the combinatorial structure of $\Delta$: combinatorially equivalent polytopes, which have the same face poset (and in particular the same number of vertices), will determine the same zero set.  As we observed in Section~\ref{S:toric}, the action of $(\mathbb{C}^*)^{q-n}$ is given by the kernel of $\mathrm{Mat}_\Delta$ as an integer matrix.  We therefore deduce:

\begin{lemma}\label{L:vertexmatrix}
Let $\Delta$ and $\Gamma$ be combinatorially equivalent reflexive polytopes, and let $R$ and $T$ be the fans over the faces of $\Delta$ and $\Gamma$ respectively.  Fix an ordering of the vertices that respects the combinatorial equivalence, and let $M_\Delta$ and $M_\Gamma$ be the matrices with columns given by the vertices in this ordering.  Suppose the kernels of $\mathrm{Mat}_\Delta$ and $\mathrm{Mat}_\Gamma$ are the same submodule of $\mathbb{Z}^q$.  Then there exists a toric variety $W$ and finite abelian subgroups of the torus $A$ and $B$ such that $V(R) = W/A$ and $V(T) = W/B$.
\end{lemma}

\begin{proof}
Let $W = (\mathbb{C}^q - Z(R))/(\mathbb{C}^*)^{q-n}$, with the action of $(\mathbb{C}^*)^{q-n}$ given by the kernel of $\mathrm{Mat}_\Delta$.  Then the result follows immediately.
\end{proof}

We searched the database of 4319 three-dimensional reflexive polytopes for polar dual pairs of polytopes satisfying the hypotheses of Lemma~\ref{L:vertexmatrix}.  Combinatorial equivalence is a fairly weak criterion: 1299 of the 4319 three-dimensional reflexive polytopes are combinatorially equivalent to their polar duals.  The kernel condition is satisfied far less often: we found 79 self-dual polytopes (including 6 simplices), and 26 pairs of non-self-dual reflexive polytopes with this property.  The simplices correspond to Greene-Plesser mirrors in weighted projective spaces, and certain finite quotients of these mirrors.  We are particularly interested in new, non-weighted projective space examples.  Of the 26 pairs of non-self-dual polytopes that we identified, 5 were not pairs of simplices.  These are polytopes 3, 10, 433, 436, 753 and their polar duals, using the numbering in the SageMath database \cite{sage}.  Polytopes 3, 753, and their polar duals yield vertex matrices with the same kernels; we classify these polytopes as Group~I.  Similarly, polytopes 10, 433, 436, and their polar duals also have vertex matrices with the same kernel; we classify these polytopes as Group~II.  Geometrically, one may write the toric variety $V_\Sigma$ for any polytope in Group~I as a resolution of a finite quotient of the toric variety determined by polytope 3, and one may write the toric variety $V_\Sigma$ for any polytope in Group~II as a resolution of a finite quotient of the toric variety determined by polytope 10.  We illustrate Group~I and II in Figures~\ref{F:reflexiveGroupI} and \ref{F:reflexiveGroupII}.

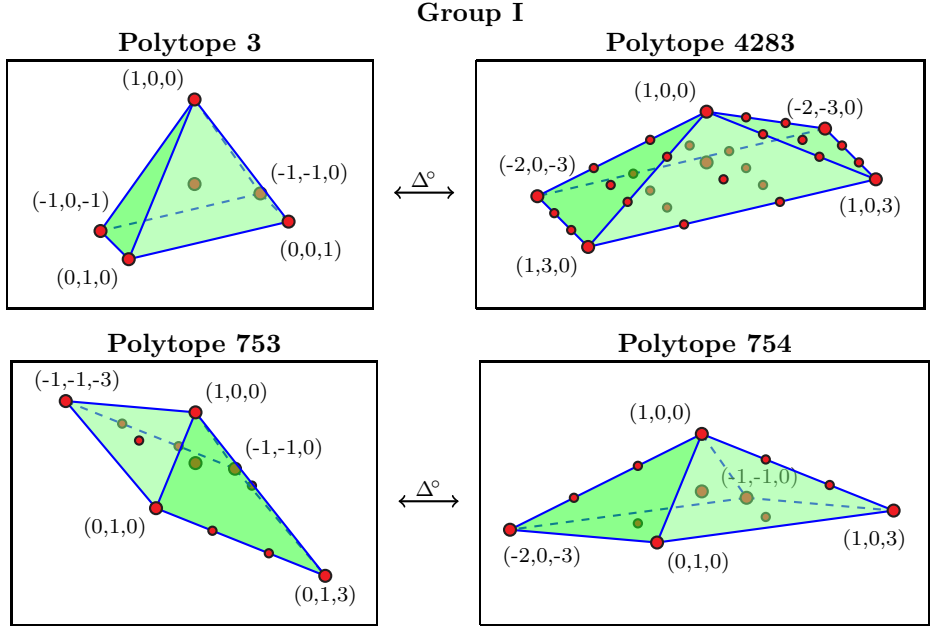
\begin{figure}[htb]
\textbf{Group I}\\
%============================================================================== class 3 and 4283
\begin{tabular}{|C{4.5cm} |C{1cm}| C{5.6cm}|}
\multicolumn{1}{c}{\textbf{Polytope 3}}& \multicolumn{1}{c}{} &
\multicolumn{1}{c}{\textbf{Polytope 4283}}\\ \cline{1-1}\cline{3-3}
\begin{tikzpicture}[view1,scale=2.5]
%\drawaxis
\draworigin
% back
%\fill[shade4] (0,1,0) -- (0,0,1) -- (-1,-1,0) -- (-1,0,-1) --cycle;
%\fill[shade1] (1,0,0) -- (-1,-1,0) -- (-1,0,-1) -- cycle;
%\fill[shade1] (1,0,0) -- (-1,-1,0) -- (0,0,1) -- cycle;
\draw[back] (-1,-1,0) -- (-1,0,-1);
\draw[back] (-1,-1,0) -- (0,0,1);
\draw[back] (1,0,0) -- (-1,-1,0);
\node[vertex] at (-1,-1,0) {};
% facets
\fill[shade3] (1,0,0) -- (0,1,0) -- (-1,0,-1) -- cycle;
\fill[shade1] (1,0,0) -- (0,1,0) -- (0,0,1) -- cycle;
% front
\draw[edge] (1,0,0) -- (0,1,0);
\draw[edge] (1,0,0) -- (0,0,1);
\draw[edge] (0,1,0) -- (0,0,1);
\draw[edge] (1,0,0) -- (-1,0,-1);
\draw[edge] (0,1,0) -- (-1,0,-1);
% vertices
\node[vertex] at (1,0,0) {};
\node[vertex] at (0,1,0) {};
\node[vertex] at (0,0,1) {};
\node[vertex] at (-1,0,-1) {};
% labels
\node[label, above left] at (1,0,0) {(1,0,0)};
\node[label, below left] at (0,1,0) {(0,1,0)};
\node[label, above right] at (-1,-1,0) {(-1,-1,0)};
\node[label, below right] at (-0.2,0,0.8) {(0,0,1)};
\node[label, above left] at (-0.8,0,-0.8) {(-1,0,-1)};
\end{tikzpicture}& \pdual &
\begin{tikzpicture}[view1,scale=1.5]
%\drawaxis
\draworigin
% back
\draw[back] (-2,-3,0) -- (-2,0,-3);
    \node[point] at (-2,-2,-1) {};
    \node[point] at (-2,-1,-2) {};
%\fill[shade1] (1,0,0) -- (-2,-3,0) -- (-2,0,-3) -- cycle;
    \node[point] at (-1,-1,-1) {};
%\fill[shade1] (1,3,0) -- (-2,-3,0) -- (1,0,3) -- (-2,0,-3) -- cycle;
    \node[point] at (0,1,0) {};
    \node[point] at (-1,0,-1) {};
    \node[point] at (0,0,1) {};
    \node[point] at (-1,-1,0) {};
% facets
\fill[shade3] (1,0,0) -- (1,3,0) -- (-2,0,-3) -- cycle;
    \node[point] at (0,1,-1) {};
\fill[shade1] (1,0,0) -- (1,3,0) -- (1,0,3) -- cycle;
    \node[point] at (1,1,1) {};
\fill[shade2] (1,0,0) -- (-2,-3,0) -- (1,0,3) -- cycle;
    \node[point] at (0,-1,1) {};
% front
\draw[edge] (1,0,0) -- (1,3,0);
    \node[point] at (1,1,0) {};
    \node[point] at (1,2,0) {};
\draw[edge] (1,0,0) -- (-2,-3,0);
    \node[point] at (0,-1,0) {};
    \node[point] at (-1,-2,0) {};
\draw[edge] (1,0,0) -- (1,0,3);
    \node[point] at (1,0,1) {};
    \node[point] at (1,0,2) {};
\draw[edge] (1,3,0) -- (1,0,3);
    \node[point] at (1,2,1) {};
    \node[point] at (1,1,2) {};
\draw[edge] (-2,-3,0) -- (1,0,3);
    \node[point] at (-1,-2,1) {};
    \node[point] at (0,-1,2) {};
\draw[edge] (1,0,0) -- (-2,0,-3);
    \node[point] at (0,0,-1) {};
    \node[point] at (-1,0,-2) {};
\draw[edge] (1,3,0) -- (-2,0,-3);
    \node[point] at (0,2,-1) {};
    \node[point] at (-1,1,-2) {};
% vertices
\node[vertex] at (1,0,0) {};
\node[vertex] at (1,3,0) {};
\node[vertex] at (-2,-3,0) {};
\node[vertex] at (1,0,3) {};
\node[vertex] at (-2,0,-3) {};
% labels
\node[label, above left] at (1,0,0) {(1,0,0)};
\node[label, below left] at (1,3,0) {(1,3,0)};
\node[label, above] at (-2,-3,0) {(-2,-3,0)};
\node[label, below] at (0.8,0,2.9) {(1,0,3)};
\node[label, above] at (-1.8,0,-3) {(-2,0,-3)};
\end{tikzpicture}\\ \cline{1-1}\cline{3-3}
\end{tabular}
\vspace{0.25cm}

%============================================================================== class 753 and 754
\begin{tabular}{|C{4.5cm} |C{1cm}| C{5.6cm}|}
\multicolumn{1}{c}{\textbf{Polytope 753}}& \multicolumn{1}{c}{} &
\multicolumn{1}{c}{\textbf{Polytope 754}}\\ \cline{1-1}\cline{3-3}
\begin{tikzpicture}[view1,scale=1.5]
%\drawaxis
\draworigin
% back
%\fill[shade1] (0,1,0) -- (-1,-1,0) -- (0,1,3) -- (-1,-1,-3) -- cycle;
%\fill[shade1] (1,0,0) -- (-1,-1,0) -- (-1,-1,-3) -- cycle;
%\fill[shade1] (1,0,0) -- (-1,-1,0) -- (0,1,3) -- cycle;
    \node[point] at (0,0,1) {};
\draw[back] (-1,-1,0) -- (0,1,3);
\draw[back] (-1,-1,0) -- (-1,-1,-3);
    \node[point] at (-1,-1,-1) {};
    \node[point] at (-1,-1,-2) {};
\draw[back] (1,0,0) -- (-1,-1,0);
\node[vertex] at (-1,-1,0) {};
% facets
\fill[shade1] (1,0,0) -- (0,1,0) -- (-1,-1,-3) -- cycle;
    \node[point] at (0,0,-1) {};
\fill[shade3] (1,0,0) -- (0,1,0) -- (0,1,3) -- cycle;
% front
\draw[edge] (1,0,0) -- (0,1,0);
\draw[edge] (1,0,0) -- (0,1,3);
\draw[edge] (0,1,0) -- (0,1,3);
    \node[point] at (0,1,1) {};
    \node[point] at (0,1,2) {};
\draw[edge] (1,0,0) -- (-1,-1,-3);
\draw[edge] (0,1,0) -- (-1,-1,-3);
% vertices
\node[vertex] at (1,0,0) {};
\node[vertex] at (0,1,0) {};
\node[vertex] at (0,1,3) {};
\node[vertex] at (-1,-1,-3) {};
% labels
\node[label, above right] at (1,0,0) {(1,0,0)};
\node[label, below left] at (0,1,0) {(0,1,0)};
\node[label, above right] at (-1,-1,0) {(-1,-1,0)};
\node[label, below] at (0,1,3) {(0,1,3)};
\node[label, above] at (-0.9,-1,-2.8) {(-1,-1,-3)};
\end{tikzpicture}& \pdual &
\begin{tikzpicture}[view1,scale=1.7]
%\drawaxis
\draworigin
% back
%\fill[shade1] (0,1,0) -- (-1,-1,0) -- (1,0,3) -- (-2,0,-3) -- cycle;
    \node[point] at (-1,0,-1) {};
    \node[point] at (0,0,1) {};
%\fill[shade1] (1,0,0) -- (-1,-1,0) -- (-2,0,-3) -- cycle;
%\fill[shade1] (1,0,0) -- (-1,-1,0) -- (1,0,3) -- cycle;    
\draw[back] (1,0,0) -- (-1,-1,0);
\draw[back] (-1,-1,0) -- (-2,0,-3);
\draw[back] (-1,-1,0) -- (1,0,3);
\node[vertex] at (-1,-1,0) {};
\node[label, above] at (-0.9,-1,0.2) {(-1,-1,0)};
% facets
\fill[shade3] (1,0,0) -- (0,1,0) -- (-2,0,-3) -- cycle;
\fill[shade1] (1,0,0) -- (0,1,0) -- (1,0,3) -- cycle;
% front
\draw[edge] (1,0,0) -- (0,1,0);
\draw[edge] (1,0,0) -- (1,0,3);
    \node[point] at (1,0,1) {};
    \node[point] at (1,0,2) {};
\draw[edge] (0,1,0) -- (1,0,3);
\draw[edge] (1,0,0) -- (-2,0,-3);
    \node[point] at (0,0,-1) {};
    \node[point] at (-1,0,-2) {};
\draw[edge] (0,1,0) -- (-2,0,-3);
% vertices
\node[vertex] at (1,0,0) {};
\node[vertex] at (0,1,0) {};
\node[vertex] at (1,0,3) {};
\node[vertex] at (-2,0,-3) {};
% labels
\node[label, above left] at (1,0,0) {(1,0,0)};
\node[label, below right] at (0,1,0) {(0,1,0)};
\node[label, below] at (0.7,0,2.7) {(1,0,3)};
\node[label, below] at (-1.9,0,-2.5) {(-2,0,-3)};
\end{tikzpicture}\\ \cline{1-1}\cline{3-3}
\end{tabular}
\caption{Two reflexive polytopes define toric varieties related by quotient maps}
\label{F:reflexiveGroupI}
\end{figure}
\vspace{0.125cm}

\begin{figure}[htb]
\textbf{Group II}\\
%============================================================================== class 10 and 4314
\begin{tabular}{|C{4.5cm} |C{1cm}| C{5.6cm}|}
\multicolumn{1}{c}{\textbf{Polytope 10}}& \multicolumn{1}{c}{} &
\multicolumn{1}{c}{\textbf{Polytope 4314}}\\ \cline{1-1}\cline{3-3}
\begin{tikzpicture}[view1,scale=2]
%\drawaxis
\draworigin
% back
%\fill[shade1] (1,0,0) -- (-2,-1,0) -- (0,0,1) -- cycle;
%\fill[shade1] (1,0,0) -- (-2,-1,0) -- (-2,0,-1) -- cycle;
\draw[back] (1,0,0) -- (-2,-1,0);
% facets
\fill[shade3] (0,0,1) -- (-2,-1,0) -- (-2,0,-1) -- (0,1,0) -- cycle;
    \node[point] at (-1,0,0) {};
\fill[shade1] (1,0,0) -- (0,1,0) -- (0,0,1) -- cycle;
\fill[shade4] (1,0,0) -- (0,1,0) -- (-2,0,-1) -- cycle;
% front
\draw[edge] (1,0,0) -- (0,1,0);
\draw[edge] (1,0,0) -- (0,0,1);
\draw[edge] (0,1,0) -- (0,0,1);
\draw[edge] (-2,-1,0) -- (0,0,1);
\draw[edge] (1,0,0) -- (-2,0,-1);
\draw[edge] (0,1,0) -- (-2,0,-1);
\draw[edge] (-2,-1,0) -- (-2,0,-1);
% vertices
\node[vertex] at (1,0,0) {};
\node[vertex] at (0,1,0) {};
\node[vertex] at (-2,-1,0) {};
\node[vertex] at (0,0,1) {};
\node[vertex] at (-2,0,-1) {};
% labels
\node[label, above left] at (1,0,0) {(1,0,0)};
\node[label, above left] at (0,1,0) {(0,1,0)};
\node[label, below right] at (-2,-1,0) {(-2,-1,0)};
\node[label, above right] at (0,0,1) {(0,0,1)};
\node[label, below left] at (-2,0,-1) {(-2,0,-1)};
\end{tikzpicture}& \pdual &
\begin{tikzpicture}[view1,scale=1.5]
%\drawaxis
\draworigin
% back
%\fill[shade1] (1,2,0) -- (-3,-2,0) -- (1,2,4) -- (-3,-2,-4) -- cycle;
    \node[point] at (-2,-1,-2) {};
    \node[point] at (-2,-1,-1) {};
    \node[point] at (-2,-1,0) {};
    \node[point] at (-1,0,-1) {};
    \node[point] at (-1,0,0) {};
    \node[point] at (-1,0,1) {};
    \node[point] at (0,1,0) {};
    \node[point] at (0,1,1) {};
    \node[point] at (0,1,2) {};
%\fill[shade1] (1,0,0) -- (-3,-2,0) -- (-3,-2,-4) -- cycle;
    \node[point] at (-1,-1,-1) {};
%\fill[shade1] (1,0,0) -- (-3,-2,0) -- (1,2,4) -- cycle;
    \node[point] at (0,0,1) {};
\draw[back] (-3,-2,0) -- (-3,-2,-4);
    \node[point] at (-3,-2,-1) {};
    \node[point] at (-3,-2,-2) {};
    \node[point] at (-3,-2,-3) {};
\draw[back] (-3,-2,0) -- (1,2,4);
    \node[point] at (-2,-1,1) {};
    \node[point] at (-1,0,2) {};
    \node[point] at (0,1,3) {};
\draw[back] (1,0,0) -- (-3,-2,0);
    \node[point] at (-1,-1,0) {};
\node[vertex] at (-3,-2,0) {};
% facets
\fill[shade3] (1,0,0) -- (1,2,0) -- (-3,-2,-4) -- cycle;
    \node[point] at (0,0,-1) {};
\fill[shade1] (1,0,0) -- (1,2,0) -- (1,2,4) -- cycle;
    \node[point] at (1,1,1) {};
% front
\draw[edge] (1,0,0) -- (1,2,0);
    \node[point] at (1,1,0) {};

\draw[edge] (1,0,0) -- (1,2,4);
    \node[point] at (1,1,2) {};
\draw[edge] (1,2,0) -- (1,2,4);
    \node[point] at (1,2,1) {};
    \node[point] at (1,2,2) {};
    \node[point] at (1,2,3) {};
\draw[edge] (1,0,0) -- (-3,-2,-4);
    \node[point] at (-1,-1,-2) {};
\draw[edge] (1,2,0) -- (-3,-2,-4);
    \node[point] at (0,1,-1) {};
    \node[point] at (-1,0,-2) {};
    \node[point] at (-2,-1,-3) {};
% vertices
\node[vertex] at (-3,-2,-4) {};
\node[vertex] at (1,0,0) {};
\node[vertex] at (1,2,0) {};
\node[vertex] at (1,2,4) {};
% labels
\node[label, above right] at (1,0,0) {(1,0,0)};
\node[label, below left] at (1,2,0) {(1,2,0)};
\node[label, above right] at (-3,-2,0.1) {(-3,-2,0)};
\node[label, below] at (1,2,4) {(1,2,4)};
\node[label, above] at (-3,-2,-4) {(-3,-2,-4)};
\end{tikzpicture}\\ \cline{1-1}\cline{3-3}
\end{tabular}
\vspace{0.25cm}

%============================================================================== class 433 and 3316
\begin{tabular}{|C{4.5cm} |C{1cm}| C{5.6cm}|}
\multicolumn{1}{c}{\textbf{Polytope 433}}& \multicolumn{1}{c}{} &
\multicolumn{1}{c}{\textbf{Polytope 3316}}\\ \cline{1-1}\cline{3-3}
\begin{tikzpicture}[view1,scale=2]
%\drawaxis
\draworigin
% back
%\fill[shade1] (0,1,0) -- (-2,-1,0) -- (0,1,2) -- (-2,-1,-2) -- cycle;
%\fill[shade1] (1,0,0) -- (-2,-1,0) -- (0,1,2) -- cycle;
%\fill[shade1] (1,0,0) -- (-2,-1,0) -- (-2,-1,-2) -- cycle;
    \node[point] at (-1,0,0) {};
\draw[back] (-2,-1,0) -- (-2,-1,-2);
    \node[point] at (-2,-1,-1) {};
\draw[back] (-2,-1,0) -- (0,1,2);
    \node[point] at (-1,0,1) {};
\draw[back] (1,0,0) -- (-2,-1,0);
\node[vertex] at (-2,-1,0) {};
% facets
\fill[shade3] (1,0,0) -- (0,1,0) -- (-2,-1,-2) -- cycle;
\fill[shade1] (1,0,0) -- (0,1,0) -- (0,1,2) -- cycle;
% front
\draw[edge] (1,0,0) -- (0,1,0);
\draw[edge] (1,0,0) -- (0,1,2);
\draw[edge] (0,1,0) -- (0,1,2);
    \node[point] at (0,1,1) {};
\draw[edge] (1,0,0) -- (-2,-1,-2);
\draw[edge] (0,1,0) -- (-2,-1,-2);
    \node[point] at (-1,0,-1) {};
% vertices
\node[vertex] at (1,0,0) {};
\node[vertex] at (0,1,0) {};
\node[vertex] at (0,1,2) {};
\node[vertex] at (-2,-1,-2) {};
% labels
\node[label, above right] at (1,0,0) {(1,0,0)};
\node[label, below left] at (0,1,0) {(0,1,0)};
\node[label, above right] at (-2,-1,0.2) {(-2,-1,0)};
\node[label, below] at (0,1,2) {(0,1,2)};
\node[label, above] at (-1.7,-1,-2.2) {(-2,-1,-2)};
\end{tikzpicture}& \pdual &
\begin{tikzpicture}[view1,scale=1.7]
%\drawaxis
\draworigin
% back
%\fill[shade1] (1,0,0) -- (-3,-2,0) -- (1,0,2) -- cycle;
%\fill[shade1] (1,0,0) -- (1,2,0) -- (-3,0,-2) -- cycle;
%\fill[shade1] (1,0,0) -- (-3,-2,0) -- (-3,0,-2) -- cycle;
\draw[back] (1,0,0) -- (-3,0,-2);
    \node[point] at (-1,0,-1) {};
\draw[back] (1,0,0) -- (-3,-2,0);
    \node[point] at (-1,-1,0) {};
% facets
\fill[shade3] (1,2,0)  -- (1,0,2) -- (-3,-2,0)-- (-3,0,-2) -- cycle;
    \node[point] at (0,1,0) {};
    \node[point] at (-2,0,-1) {};
    \node[point] at (-1,0,0) {};
    \node[point] at (0,0,1) {};
    \node[point] at (-2,-1,0) {};
\fill[shade1] (1,0,0) -- (1,2,0) -- (1,0,2) -- cycle;
% front
\draw[edge] (1,0,0) -- (1,2,0);
    \node[point] at (1,1,0) {};
\draw[edge] (1,0,0) -- (1,0,2);
    \node[point] at (1,0,1) {};
\draw[edge] (1,2,0) -- (1,0,2);
    \node[point] at (1,1,1) {};
\draw[edge] (-3,-2,0) -- (1,0,2);
    \node[point] at (-1,-1,1) {};
\draw[edge] (1,2,0) -- (-3,0,-2);
    \node[point] at (-1,1,-1) {};
\draw[edge] (-3,-2,0) -- (-3,0,-2);
    \node[point] at (-3,-1,-1) {};
% vertices
\node[vertex] at (1,0,0) {};
\node[vertex] at (1,2,0) {};
\node[vertex] at (-3,-2,0) {};
\node[vertex] at (1,0,2) {};
\node[vertex] at (-3,0,-2) {};
% labels
\node[label, above left] at (1,0,0) {(1,0,0)};
\node[label, above left] at (1,2,0) {(1,2,0)};
\node[label, below right] at (-3,-2,0) {(-3,-2,0)};
\node[label, above right] at (1,0,2) {(1,0,2)};
\node[label, below] at (-3,0,-2) {(-3,0,-2)};
\end{tikzpicture}\\ \cline{1-1}\cline{3-3}
\end{tabular}
\vspace{0.25cm}

%============================================================================== class 436 and 3321
\begin{tabular}{|C{4.5cm} |C{1cm}| C{5.6cm}|}
\multicolumn{1}{c}{\textbf{Polytope 436}}& \multicolumn{1}{c}{} &
\multicolumn{1}{c}{\textbf{Polytope 3321}}\\ \cline{1-1}\cline{3-3}
\begin{tikzpicture}[view1,scale=1.8]
%\drawaxis
\draworigin
% back
%\fill[shade1] (1,0,0) -- (-2,-1,0) -- (-3,0,-2) -- cycle;
%\fill[shade1] (1,0,0) -- (-2,-1,0) -- (1,0,2) -- cycle;
\draw[back] (1,0,0) -- (-2,-1,0);
% facets
\fill[shade4] (0,1,0) -- (-3,0,-2) -- (-2,-1,0) -- (1,0,2)-- cycle;
    \node[point] at (-2,0,-1) {};
    \node[point] at (-1,0,0) {};
    \node[point] at (0,0,1) {};
\fill[shade1] (1,0,0) -- (0,1,0) -- (1,0,2) -- cycle;
\fill[shade2] (1,0,0) -- (0,1,0) -- (-3,0,-2) -- cycle;
% front
\draw[edge] (1,0,0) -- (0,1,0);
\draw[edge] (1,0,0) -- (1,0,2);
    \node[point] at (1,0,1) {};
\draw[edge] (0,1,0) -- (1,0,2);
\draw[edge] (-2,-1,0) -- (1,0,2);
\draw[edge] (1,0,0) -- (-3,0,-2);
    \node[point] at (-1,0,-1) {};
\draw[edge] (0,1,0) -- (-3,0,-2);
\draw[edge] (-2,-1,0) -- (-3,0,-2);
% vertices
\node[vertex] at (1,0,0) {};
\node[vertex] at (0,1,0) {};
\node[vertex] at (-2,-1,0) {};
\node[vertex] at (1,0,2) {};
\node[vertex] at (-3,0,-2) {};
% labels
\node[label, above left] at (1,0,0) {(1,0,0)};
\node[label, below right] at (0.1,1.1,0) {(0,1,0)};
\node[label, below right] at (-2,-1,0) {(-2,-1,0)};
\node[label, above] at (1.1,0,1.9) {(1,0,2)};
\node[label, below] at (-3,0,-1.8) {(-3,0,-2)};
\end{tikzpicture}& \pdual &
\begin{tikzpicture}[view1,scale=1.55]
%\drawaxis
\draworigin
% back
%\fill[shade1] (0,1,0) -- (-2,-1,0) -- (0,1,4) -- (-2,-1,-4) -- cycle;
    \node[point] at (-1,0,-1) {};
    \node[point] at (-1,0,0) {};
    \node[point] at (-1,0,1) {};
%\fill[shade1] (1,0,0) -- (-2,-1,0) -- (-2,-1,-4) -- cycle;
%\fill[shade1] (1,0,0) -- (-2,-1,0) -- (0,1,4) -- cycle;
    \node[point] at (0,0,1) {};
\draw[back] (1,0,0) -- (-2,-1,0);
\draw[back] (-2,-1,0) -- (0,1,4);
    \node[point] at (-1,0,2) {};    
\draw[back] (-2,-1,0) -- (-2,-1,-4);
    \node[point] at (-2,-1,-1) {};
    \node[point] at (-2,-1,-2) {};
    \node[point] at (-2,-1,-3) {};
\node[vertex] at (-2,-1,0) {};
\node[label, above] at (-1.9,-1,0.3) {(-2,-1,0)};
% facets
\fill[shade3] (1,0,0) -- (0,1,0) -- (-2,-1,-4) -- cycle;
    \node[point] at (0,0,-1) {};
\fill[shade1] (1,0,0) -- (0,1,0) -- (0,1,4) -- cycle;
% front
\draw[edge] (1,0,0) -- (-2,-1,-4);
\draw[edge] (1,0,0) -- (0,1,0);
\draw[edge] (1,0,0) -- (0,1,4);
\draw[edge] (0,1,0) -- (0,1,4);
    \node[point] at (0,1,1) {};
    \node[point] at (0,1,2) {};
    \node[point] at (0,1,3) {};
\draw[edge] (0,1,0) -- (-2,-1,-4);
    \node[point] at (-1,0,-2) {};
% vertices
\node[vertex] at (1,0,0) {};
\node[vertex] at (0,1,0) {};
\node[vertex] at (0,1,4) {};
\node[vertex] at (-2,-1,-4) {};
% labels
\node[label, above right] at (1,0,0) {(1,0,0)};
\node[label, below left] at (0,1,0) {(0,1,0)};
\node[label, below left] at (0,1,4.3) {(0,1,4)};
\node[label, above right] at (-2,-1,-4.2) {(-2,-1,-4)};
\end{tikzpicture}\\ \cline{1-1}\cline{3-3}
\end{tabular}
\caption{Three reflexive polytopes define toric varieties related by quotient maps}
\label{F:reflexiveGroupII}
\end{figure}
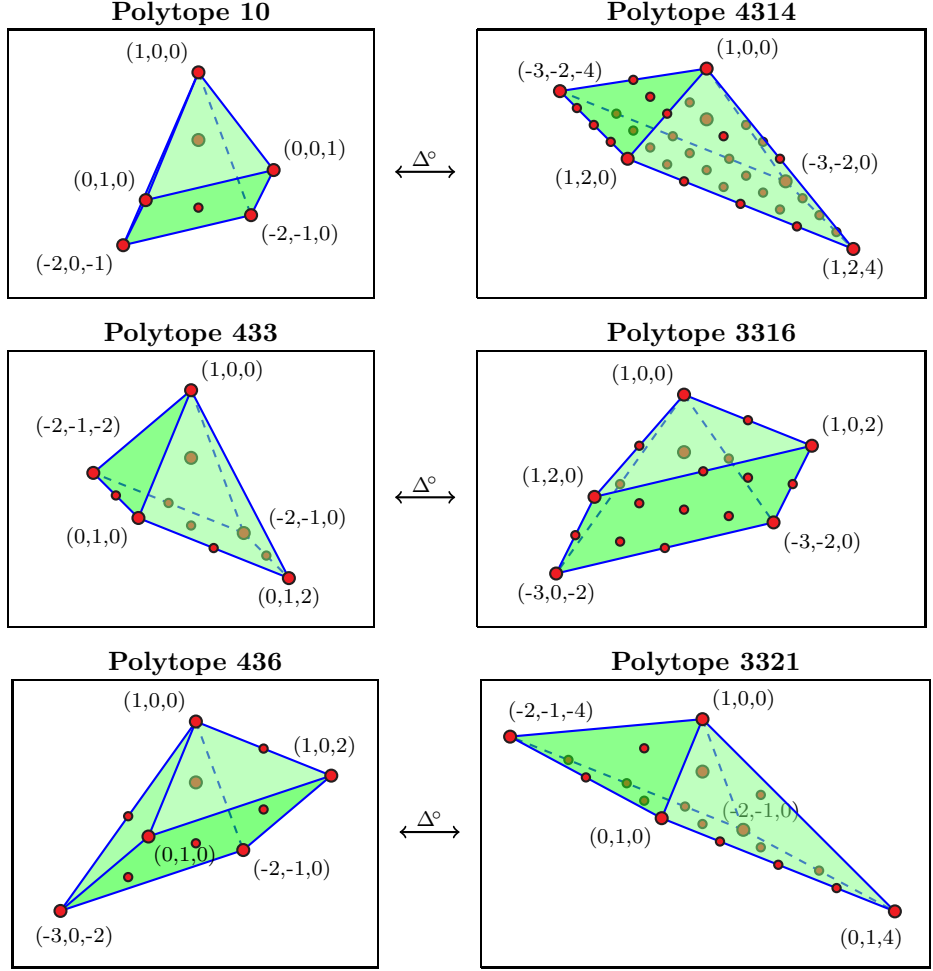

We use Sage Mathematics Software (\cite{sage}) to count the points on $X_t$ and $Y_t$ for fields $\F_p$ of order 5,7, or 11, where $t = 0, \dots, p-1$.  In each case, we generate an explicit list of the points on $V_\Sigma$ and substitute each point into $f_{\Delta, t}$.  We then compare the point count for $f_{\Delta^\circ, t}$.  On a current medium-to-high performance desktop computer, these computations require a minimum of several hours of processor time for each pair $\Delta$ and $\Delta^\circ$.  We find that for all smooth $X_t$ and $Y_t$, and $p=5, 7,$ or $11$,
$$\#X_t(\F_p) \equiv \#Y_t(\F_p) \pmod{p}$$
when $\Delta$ is in Group~I or Group~II.  Not only do we observe experimental evidence for strong arithmetic mirror symmetry in these cases, we observe further arithmetic relationships between K3 surfaces determined by members of the same group.  Indeed, suppose $\Delta$ and $\Gamma$ are both members of either Group~I or Group~II, let $X_t$ be given by $f_{\Delta, t}$, and let $W_t$ be given by $f_\Gamma(t)$.  Suppose $X_t$ and $W_t$ are smooth.  Then our experimental evidence shows
$$ \#X_t(\F_p) \equiv \#W_t(\F_p) \pmod{p}$$
for fields $\F_p$ of order 5,7, or 11, where $t = 0, \dots, p-1$.

We conjecture that strong arithmetic mirror symmetry will hold for similar toric examples in higher dimensions.

\begin{conjecture}
Let $\Delta$ and $\Gamma$ be a pair of reflexive polytopes satisfying the hypotheses of Lemma~\ref{L:vertexmatrix}, and let $X_t$ and $W_t$ be the hypersurfaces determined by $f_{\Delta, t}$ and $f_{\Gamma}(t)$, respectively.  Then for any rational $t$ where $X_t$ and $Y_t$ are smooth, strong arithmetic mirror symmetry holds:
$$\#X_t(\F_q) \equiv \#W_t(\F_q) \pmod{q}.$$
\end{conjecture}

\section{Picard-Fuchs equations}

There is a moduli space of lattice-polarized K3 surfaces.  Within this moduli space lie one-parameter families of generically Picard rank 19 K3 surfaces.  Because our toric construction yields one-parameter families in a combinatorially natural way, we expect that these are generically Picard rank 19 families.  We say a family of K3 surfaces is \emph{lattice-polarized} if there exists a lattice $M$ such that $M \hookrightarrow \mathrm{Pic}(X)$ for every K3 surface $X$ in the family.  The dimension of $M$ places a lower bound on the Picard rank of each member of the family.  Lattice-polarized Picard rank 19 families have special structure: for example, Picard rank 19 families polarized by the lattice $U \oplus E_8 \oplus E_8 + (-2n)$, where $U$ is a unimodular lattice of signature $(1,1)$ and $E_8$ is taken to be negative definite, are geometrically related to a product of two elliptic curves with an $n$-isogeny, via the \emph{Shioda-Inose structure}.  More generally, Dolgachev described a mirror map for lattice-polarized families of K3 surfaces in \cite{Dolgachev}.

Computing the general Picard lattice for a specific parametrized family of K3 surfaces is typically a difficult and ad-hoc process.  Rather than doing so directly, we will use differential equations to extract information about the geometric structure of our families, following the strategy of \cite{dor}.

 A \emph{period} is the integral of a differential form with respect to a specified homology class. 
 Periods of holomorphic forms encode the complex structure of varieties.
 The \emph{Picard-Fuchs differential equation} of a family of varieties is a differential equation that describes the way the value of a period changes as we move through the family.
Solutions to Picard-Fuchs equations for holomorphic forms on Calabi-Yau varieties can be used to define a mirror map, which relates the moduli space of complex deformations of a variety to a moduli space of complexified K\"{a}hler deformations of the mirror variety.

We use the Griffiths-Dwork technique to compute the Picard-Fuchs equation for the holomorphic form on the pencil $f_{\Delta, t}$ for each of the reflexive polytopes $\Delta$ in Group I and Group II, following the method outlined in \cite{klmsw}.  The quotient maps which relate different toric varieties determined by polytopes in Group~I or Group~II preserve the holomorphic forms on the K3 surfaces; thus, the Picard-Fuchs equation will be the same for every member of a group.  We obtain the pair of third-order differential equations given in Table~\ref{Ta:picardfuchs}.  By \cite{dor}, because the Picard-Fuchs equations are third-order, the Picard rank of a general K3 surface in each family must be 19, as we expected.

\begin{table}[htb]
\begin{tabular}{|c|c|}
\hline
\textbf{Group} & \textbf{Picard-Fuchs Equation}\\
\hline 
I & $\displaystyle \p{3t} \mathcal{P} + \p{7t^2} \mathcal{P'} + \p{6t^3 + 162} \mathcal{P''} + \p{t^4 + 108t} \mathcal{P'''} = 0$ \\
II & $\displaystyle \p{t} \mathcal{P} + \p{7t^2} \mathcal{P'} + \p{6t^3} \mathcal{P''} + \p{t^4 - 256} \mathcal{P'''} = 0$ \\
\hline
\end{tabular}
\caption{Picard-Fuchs equations}\label{Ta:picardfuchs}
\end{table}

As we noted earlier, the Picard-Fuchs equation for the Fermat pencil determines its point count $\pmod{q}$.  One might speculate that the Picard-Fuchs equations are enough to predict the relationships between point counts for members of each group.

\begin{question}
Let $V_t$ and $W_t$ be pencils of Calabi-Yau $n$-folds, and suppose the Picard-Fuchs equation of $V_t$ is the same as the Picard-Fuchs equation of $W_t$.  Is
$$\#V_t(\F_q) \equiv \#W_t(\F_q) \pmod{q}$$
when $V_t$ and $W_t$ are smooth?
\end{question}

We note that the Picard-Fuchs equations computed in Table~\ref{Ta:picardfuchs} have a special structure: we can relate these equations to lower-order differential equations which appear to be Picard-Fuchs equations of elliptic curves.

\begin{definition}
Let $L(y)$ be a homogeneous linear differential equation with coefficients in ${\mathbb{C}}(t)$.  Then there exists a homogeneous linear differential equation $M(y) = 0$ with coefficients in $\mathbb{C}(t)$, such that the solution space of $M(y)$ is the $\mathbb{C}$-span of 
$$\{\nu_1 \nu_2 \ | \ L(\nu_1) = 0 \ \mbox{and} \ L(\nu_2) = 0 \} \ .$$
We say the differential equation $M(y)$ is the \emph{symmetric square} of $L$.
\end{definition}

In particular, the symmetric square of the second-order linear, homogeneous differential equation
$$
a_2\derive{\omega}{2} + a_1\frac{\partial \omega}{\partial t} + a_0\omega = 0
$$
is
\begin{align}\label{E:symmetricSquare}
\begin{split}
a_2^2\derive{\omega}{3} + 3a_1a_2\derive{\omega}{2} + (4a_0a_2+2a_1^2+a_2a_1'-a_1a_2')\frac{\partial \omega}{\partial t} \\
 + (4a_0a_1 + 2a_0'a_2-2a_0a_2')\omega = 0 
\end{split}
\end{align}
where primes denote derivatives with respect to $t$.

\begin{theorem}\cite[Theorem 5]{dor}\label{T:DSymmSquare} The Picard-Fuchs equation of a family of Picard rank-$19$ lattice-polarized K3 surfaces is a third-order ordinary differential equation which can be written as the symmetric square of a second-order homogeneous linear Fuchsian differential equation. 
\end{theorem}

The Picard-Fuchs equations for the families in Group~I and Group~II are symmetric squares.  We give the corresponding symmetric square roots, which are second-order differential equations, in Table~\ref{Ta:squareroot}.

\begin{table}[htb]
\begin{tabular}{|c|c|}
\hline
\textbf{Group} & \textbf{Symmetric Square Root}\\
\hline 
I & $\displaystyle \p{\frac{t^2}{4}} \mathcal{P} + \p{2t^3 + 54} \mathcal{P'} + \p{t^4 + 108t} \mathcal{P''} = 0$\\
II & $\displaystyle \p{\frac{t^2}{4}} \mathcal{P} + \p{2 t^3} \mathcal{P}' + \p{t^4 - 256} \mathcal{P}'' = 0$ \\
\hline
\end{tabular}
\caption{Symmetric square roots of Picard-Fuchs equations}\label{Ta:squareroot}
\end{table}

\begin{definition}
The \emph{projective normal form} of an order $k$ Fuchsian differential equation
$$f_0\of{t} \mathcal{P} + f_1\of{t} \mathcal{P'} + \dots + f_{k-1}\of{t} \mathcal{P}^{\p{k-1}} + \mathcal{P}^{\p{k}}= 0$$
is the unique order $k$ Fuchsian differential equation without a $\p{k-1}$ order derivative
$$g_0\of{t} \mathcal{R} + g_1\of{t} \mathcal{R'} + \dots + g_{k-2}\of{t} \mathcal{R}^{\p{k-2}} + \mathcal{R}^{(k)} ,\hspace{0.5cm} g_i\of{t} \in \C\of{t}$$
obtained by scaling the fundamental solutions by the $k$th root of the Wronskian.
\end{definition}

In \cite{dor}, Doran showed that the mirror map of a one-parameter family of Calabi-Yau varieties is determined by the projective normal form of its Picard-Fuchs equation. 

In the case of a second-order Fuchsian differential equation
$$f_0\of{t} \mathcal{P} + f_1\of{t} \mathcal{P'} + \mathcal{P''} = 0$$
the projective normal form is simply
$$\p{f_0\of{t} - \frac{1}{2} f_1'\of{t} - \frac{1}{4} f_1\of{t}^2} \mathcal{R} + \mathcal{R}'' = 0.$$
We compute the projective normal forms for our symmetric square roots in Table~\ref{Ta:pnf}.

\begin{table}[htb]
\begin{tabular}{|c|c|}
\hline
\textbf{Group} & \textbf{Projective Normal Form}\\
\hline 
I & $\displaystyle \p{\frac{t^6 - 540 t^3 + 8748}{4 t^2 \p{t^3 + 108}^2}} \mathcal{R} + \mathcal{R''} =  0$\\
II & $\displaystyle \p{\frac{t^2 \p{t^4 + 2816}}{4 \p{t^2 + 16}^2 \p{t+4}^2 \p{t-4}^2}} \mathcal{R} + \mathcal{R''} = 0$ \\
\hline
\end{tabular}
\caption{Projective normal forms of symmetric square roots}\label{Ta:pnf}
\end{table}

We saw in \S~\ref{S:elliptic} that mirror relationships for elliptic curves can be described as isogenies, which preserve arithmetic information.  One might ask whether the strong arithmetic mirror symmetry relationship we observe for members of Group~I and Group~II ultimately derives from a property of elliptic curves: is there a geometric construction relating our families of K3 surfaces to families of elliptic curves with Picard-Fuchs equations described above?

\bibliographystyle{amsalpha}

\end{document}